\documentclass[brochure,english,12pt]{bourbaki}

\usepackage[OT2,T1]{fontenc}
\usepackage{lmodern,amssymb,bm}

\usepackage[french]{babel}
\usepackage[utf8]{inputenc}

\usepackage{url}
\usepackage{amsmath}
\usepackage{paralist}
\usepackage{colonequals}			% for nice :=
\usepackage{mathabx}			% for \widec
\usepackage{amscd}
\usepackage[all,cmtip]{xy}			% for commutative diagrams
\usepackage{verbatim}
\usepackage{mathrsfs} 			% for \mathscr (script letters)
\usepackage{cleveref}
\usepackage{enumitem}
\usepackage{moreenum}			% for enumeration via Greek letters
\usepackage[alphabetic,lite]{amsrefs} 	% for bibliography
\usepackage{stmaryrd}			% this is for \llbracket and \rrbracket
\usepackage{stackengine}

\newcommand{\bA}{\mathbb{A}}

\newcommand{\bC}{\mathbb{C}}

\newcommand{\bF}{\mathbb{F}}

\newcommand{\bP}{\mathbb{P}}
\newcommand{\bQ}{\mathbb{Q}}

\newcommand{\bZ}{\mathbb{Z}}

\newcommand{\cP}{\mathcal{P}}

\newcommand{\sL}{\mathscr{L}}

\newcommand{\sO}{\mathscr{O}}

\newcommand{\sR}{\mathscr{R}}

\DeclareMathOperator{\Base}{Base}		% Linear equivalence class
\DeclareMathOperator{\Char}{char}		% Characteristic of a field
\DeclareMathOperator{\Cl}{Cl}			% Class group
\DeclareMathOperator{\codim}{codim}	% codimension
\DeclareMathOperator{\Gal}{Gal}		% Galois group
\DeclareMathOperator{\Gr}{Gr}			% Affine Grassmannian
\DeclareMathOperator{\NS}{NS}		% The Neron--Severi group
\DeclareMathOperator{\Proj}{Proj}		% Proj of a graded ring
\DeclareSymbolFont{cyrletters}{OT2}{wncyr}{m}{n}
\DeclareMathSymbol{\Sha}{\mathalpha}{cyrletters}{"58}	% Sha
\DeclareMathOperator{\Spec}{Spec}		% Spectrum of a ring
\newcommand{\ce}{\colonequals}
\newcommand{\Div}{{\mathrm{Div}}}		% Cartier divisor associated to a generically exact complex
\newcommand{\hra}{\hookrightarrow}
\renewcommand{\i}{^{-1}}
\newcommand{\isomto}{\overset{\sim}{\longrightarrow}}

\newcommand{\ov}{\overline}
\providecommand{\p}[1]{\left(#1\right)}
\newcommand{\ra}{\rightarrow}

\newcommand{\red}{{\mathrm{red}}}		% The maximal reduced quotient
\newcommand{\sm}{\mathrm{sm}}			% smooth locus

\newcommand{\tensor}{\otimes} 			% tensor product
\newcommand{\wt}{\widetilde}

\providecommand{\abs}[1]{\left\lvert#1\right\rvert}

\providecommand{\uref}[1]{{\upshape\ref{#1}}}
\providecommand{\uS}{{\upshape\S}}
\providecommand{\f}[2]{\frac{#1}{#2}}
\renewcommand{\b}{\textbf}

\providecommand{\uscolon}{{\upshape;} }

\newcommand{\brems}{\begin{rems} \hfill \begin{enumerate}[label=\b{\thenumberingbase.},ref=\thenumberingbase]}

\newcommand{\erems}{\end{enumerate} \end{rems}}
\newcommand{\begs}{\begin{egs} \hfill \begin{enumerate}[label=\b{\thenumberingbase.},ref=\thenumberingbase]}

\newcommand{\eegs}{\end{enumerate} \end{egs}}
\newcommand{\m}{\item}
\newcommand{\bsm}{\begin{smallmatrix}}
\newcommand{\esm}{\end{smallmatrix}}
\newcommand{\blem}{\begin{lemma}}
\newcommand{\elem}{\end{lemma}}

\newcommand{\bconj}{\begin{conj}}
\newcommand{\econj}{\end{conj}}
\newcommand{\bprob}{\begin{Problem}}
\newcommand{\eprob}{\end{Problem}}

\newcommand{\bq}{\begin{Q}}
\newcommand{\eq}{\end{Q}}
\newcommand{\benum}{\begin{enumerate}[label={{\upshape(\alph*)}}]}
\newcommand{\benuma}{\begin{enumerate}[label={{\upshape(\arabic*)}}]}
\newcommand{\benumr}{\begin{enumerate}[label={{\upshape(\roman*)}}]}
\newcommand{\eenum}{\end{enumerate}}
\newcommand{\bitem}{\begin{itemize}}
\newcommand{\eitem}{\end{itemize}}
\newcommand{\bc}{\begin{comment}}
\newcommand{\ec}{\end{comment}}
\newcommand{\bd}{\begin{defn}}
\newcommand{\ed}{\end{defn}}

\newcommand{\beg}{\begin{eg}}
\newcommand{\eeg}{\end{eg}}

\newcommand{\bcl}{\begin{claim}}
\newcommand{\ecl}{\end{claim}}

\newcommand{\msk}{\medskip}

\newcommand{\x}{\text}

\providecommand{\qxq}[1]{\quad\text{#1}\quad}
\providecommand{\qx}[1]{\quad\text{#1}}

\newcommand{\tst}{\textstyle}
\newcommand{\ba}{\begin{aligned}}
\newcommand{\ea}{\end{aligned}}
\newcommand{\be}{\begin{equation}}
\newcommand{\ee}{\end{equation}}
\newcommand{\bpf}{\begin{proof}}
\newcommand{\epf}{\end{proof}}
\newcommand{\bthm}{\begin{thm}}
\newcommand{\ethm}{\end{thm}}
\newcommand{\bprop}{\begin{prop}}
\newcommand{\eprop}{\end{prop}}
\newcommand{\bcor}{\begin{cor}}
\newcommand{\ecor}{\end{cor}}
\newcommand{\brem}{\begin{rem}}
\newcommand{\erem}{\end{rem}}

\date{Avril 2021}
\bbkannee{73\textsuperscript{e} ann\'ee, 2020--2021}  % Commandes sp\'eciales
\bbknumero{1175}                                      % Commandes spéciales

\title{Reconstructing a variety from its topology}
\subtitle{after Koll\'{a}r, building on earlier work of Lieblich and Olsson}

\author{K\k{e}stutis \v{C}esnavi\v{c}ius}
\address{CNRS, Universit\'{e} Paris-Saclay, \\  Laboratoire de math\'{e}matiques d'Orsay, \\ F-91405, Orsay, France}
\email{kestutis@math.u-psud.fr}

\begin{document}

\maketitle

\section*{Introduction}

The underlying Zariski topological space $\abs{X}$ of an algebraic variety or, more generally, a scheme $X$ tends to have few open subsets in comparison to topologies that underlie structures appearing in differential geometry or geometric topology. Thus, intuitively, $\abs{X}$ is a weak invariant of $X$, and this intuition is confirmed by low-dimensional examples: for an algebraic curve $C$, the proper closed subsets of $\abs{C}$ are the finite subsets of closed points, so $\abs{C}$ does not see much beyond the cardinality of the algebraic closure of the base field. A more surprising example was constructed by Wiegand and Krauter \cite{WK81}*{Cor.~1}: for primes $p$ and $p'$, there is a homeomorphism
\[
|\bP^2_{\ov{\bF}_p}| \simeq |\bP^2_{\ov{\bF}_{p'}}|.
\]

Topological spaces that underlie schemes (resp.,~affine schemes) were, in fact, completely classified by Hochster \cite{Hoc69}*{Thm.~9}: they are the locally spectral (resp.,~the spectral) topological spaces. We recall that a topological space $T$ is \emph{spectral} if
\begin{itemize}
\m
it is quasi-compact and quasi-separated;

\m
it is \emph{sober}: each irreducible closed $T' \subset T$ is the closure $\ov{\{t\}}$ of a unique $t \in T'$;

\m
the quasi-compact open subsets form a base of the topology of $T$.
\end{itemize}
A topological space $T$ is \emph{locally spectral} if it has an open cover by spectral spaces. The topological space $\abs{X}$ of a quasi-compact and quasi-separated scheme $X$ is spectral, so Hochster's result implies that, somewhat surprisingly, $\abs{X}$ also underlies some affine scheme. For instance, for any field $k$ and any $n \ge 0$, the topological space $\abs{\bP^n_k}$ also underlies an affine scheme (which, of course,  need not be a variety over a field).

Due to the above, the recent result of Koll\'{a}r \cite{Kol20}, which is the focus of this report, came as a surprise: a projective, irreducible, normal variety $X$ over $\bC$ of dimension $\ge 4$ is uniquely determined by its topological space $\abs{X}$, see Theorem \ref{thm:main} below. A resulting general expectation in this direction is captured by the following conjecture of Koll\'{a}r.

\begin{conj}[\cite{Kol20}*{Conj.~3}] \label{main-conj}
For seminormal, geometrically irreducible varieties $X$ and $X'$ over fields $k$ and $k'$, respectively, with $\Char k = 0$ and $\dim X \ge 2$, every homeomorphism $\abs{X} \isomto \abs{X'}$ underlies a scheme isomorphism $X \isomto X'$.

%n isomorphism of varieties 
%\[
%X \tensor_{K,\, \iota} K' \isomto X' \qxq{over} K' \qxq{for some field isomorphism} \iota \colon K \isomto K'.
%\]
%for some field isomorphism $\iota \colon K \isomto K'$.
\end{conj}

%%%%%%%%%%%%%%%%%%%%%%%%%%%%%%%%

\section{Reconstruction of projective varieties} \label{sec:master}

% %%% the citations :

%Last, if we did not do so, you should change the class option (see the commented line at the
%beginning of the file) so that the \og et\fg{} are converted to \og and\fg{}.
 
%%% predefined theorems

%%% style "plain" :

The following result of Koll\'{a}r builds on previous work of Lieblich and Olsson and fully resolves Conjecture \ref{main-conj} for projective, normal varieties of dimension $\ge 4$ in characteristic $0$. In fact, it forms the foundation of credibility for a conjecture of this sort.

\begin{theo}[\cite{Kol20}*{Thm.~2}] \label{thm:main}
For normal, geometrically integral, projective varieties $X$ and $X'$ over fields $k$ and $k'$, respectively, such that either
\benumr
\m \label{main-i}
$\dim X \ge 4$ and both $k$ and $k'$ are of characteristic $0$\uscolon or

\m \label{main-ii}
$\dim X \ge 3$ and both $k$ and $k'$ are finitely generated field extensions of $\bQ$\uscolon
\eenum
every homeomorphism $\abs{X} \isomto \abs{X'}$ underlies a scheme isomorphism $X \isomto X'$.
%for some field isomorphism $\iota \colon K \isomto K'$. 
\end{theo}

\begin{rema}
Since $X$ and $X'$ are proper and geometrically integral, we have $\Gamma(X, \sO_X) \cong k$ and $\Gamma(X', \sO_{X'}) \cong k'$, so a scheme isomorphism $X \isomto X'$ amounts to a field isomorphism $\iota \colon k \isomto k'$ and an isomorphism of varieties $X \tensor_{k,\, \iota} k' \isomto X'$.
\end{rema}

We will focus on case \ref{main-i} because it already contains most of the main ideas while avoiding further technicalities of \ref{main-ii} that largely concern the Hilbert irreducibility theorem. Roughly, the proof is based on studying Weil divisors $D$ on a normal $X$: such $D$ are determined by $\abs{X}$ alone because they may be viewed as formal $\bZ$-linear combinations of the points of codimension $1$---for instance, a reduced effective divisor $D \subset X$ is the closure of a finite set of codimension $1$ points in $X$. %and a general effective divisor $D \subset X$ is the schematic image of the Artinian closed subscheme of the codimension $1$ points of $X$ associated to a formal $\bZ_{\ge 0}$-linear combination of these points. 
We will let
\[
\tst \Div(X) \ce \bigoplus_{x \in X^{(1)}} \bZ \qxq{and} \mathrm{Eff}(X) \ce \bigoplus_{x \in X^{(1)}} \bZ_{\ge 0}
\]
denote the group of all divisors (resp.,~the monoid of all effective divisors) on $X$. 

It is not clear if notions such as ampleness or linear equivalence of divisors are determined by $\abs{X}$ alone, and the crux of the argument is in exhibiting hypotheses under which they are. The ability to topologically recognize linear equivalence eventually reduces the reconstruction problem to a combinatorial recognition theorem for projective spaces in terms of incidence of their lines and points (von Staudt's fundamental theorem of projective geometry).

A divisor $D$ on $X$ is \emph{ample} if some multiple $nD$ with $n > 0$ is a Cartier divisor %(so that $D$ is \emph{$\bQ$-Cartier}) 
whose associated line bundle $\sO(nD)$ is ample. We let $\sim$ denote linear equivalence of divisors and %let 
%\[
%\Lin(D) \ce \{ D' \ge 0 \, \vert\, D' \sim D\}
%\]
%be the linear equivalence class of $D$. We 
say that divisors $D_1$ and $D_2$ on $X$ are \emph{linearly similar}, denoted by $D_1 \sim_{\mathrm{s}} D_2$, if $n_1D_1 \sim n_2 D_2$ for some nonzero integers $n_1$ and $n_2$. If this holds with $n_1 = n_2$, then we say that $D_1$ and $D_2$ are \emph{$\bQ$-linearly equivalent}, denoted by $D_1 \sim_\bQ D_2$. When we speak of reduced (resp.,~irreducible) divisors, we implicitly assume that they are also effective (resp.,~effective and reduced). With these definitions, the overall proof of Theorem \ref{thm:main} proceeds in the following stages, which successively reconstruct more and more of the structure of $X$ from the topological space $\abs{X}$, and which will be discussed individually in the indicated sections:
\[\ba
\abs{X} &\overset{\x{\uS\uref{sec:ample}}}{\leadsto} %(\abs{X}, \x{ample divisors}) \overset{\x{\uS\uref{sec:sim-ample}}}{\leadsto} 
(\abs{X}, \sim_{\mathrm{s}} \x{of irreducible ample divisors}) 
% \overset{\x{\uS\uref{sec:linear-ample}}}{\leadsto} (\abs{X}, \sim \x{of irreducible ample divisors}) \\
%& 
\overset{\x{\uS\uref{sec:linear-ample}--\uS\uref{sec:linear}}}{\leadsto} (\abs{X}, \sim \x{of effective divisors}) \overset{\x{\uS\uref{sec:recover-X}}}{\leadsto} X.
\ea \]
The last step, namely, the determination of a normal, geometrically integral, projective variety $X$ of dimension $\ge 2$ over an infinite field from its underlying topological space $\abs{X}$ equipped with the relation of linear equivalence between effective divisors on $\abs{X}$ is due to Lieblich and Olsson \cite{LO19}.  

The initial results of the preprint \cite{LO19}, although already sufficient for Theorem \ref{thm:main} above, have been sharpened and expanded in the later preprint \cite{KLOS20}. %The latter also extends Theorem \ref{main} as follows: granted that the base field $K$ is actually fixed and uncountable of characteristic $0$, our  normal, geometrically integral, projective $X$ is determined by its topological space $\abs{X}$ even when $\dim X = 2$ or $\dim X = 3$; in addition, the projectivity assumption may be relaxed to properness whenever $\dim X \ge 2$. 

%%%%%%%%%%%%%%%%%%%%%%%%%%%%%%%%%

\section{Recovering linear similarity of ample divisors} \label{sec:ample}

\textsc{Notation.} \emph{In this section, we let $X$ be a normal, geometrically integral, projective variety over a field $k$ of characteristic $0$.} %that is not a subfield of any $\ov{\bF}_p$. }

\msk

The first stage of the proof of Theorem \ref{thm:main} is the reconstruction of linear similarity of irreducible ample divisors from the topological space $\abs{X}$ alone. This requires, in particular, to be able to topologically recognize ampleness of irreducible divisors, which rests crucially on the following Lefschetz type theorem for the divisor class group.

\begin{lemm}[\cite{RS06}*{Thm.~1}] \label{lem:RS}
Suppose that $\dim X \ge 3$ and let $\sL$ be an ample line bundle on $X$ whose linear system $\Gamma(X, \sL)$ is basepoint free. For some nonempty Zariski open $U \subset \Gamma(X, \sL)$ and every effective divisor $D \subset X$ that corresponds to a $k$-point in $U$, the following restriction map is injective:
\[
\Cl(X) \hra \Cl(D).
\]
\end{lemm}

%\bpf
The cited result is sharper but only applies to the base change $X_{\ov{k}}$ to an algebraic closure $\ov{k}$. This suffices because $\Cl(X) \hra \Cl(X_{\ov{k}})$: to see this last injectivity, note that for any divisor $H$ on $X$ that represents a class in the kernel, both $\sO(H)$ and $\sO(-H)$ have nonzero global sections, which, since $X$ is projective, means that $H \sim 0$. 

For proving Theorem \ref{thm:main} for varieties of dimension $\le 4$, one needs a refinement of Lemma \ref{lem:RS} in which $X$ is a surface (and $D$ is a curve). This requires arithmetic inputs, notably a theorem of N\'{e}ron from \cite{Ner52} on specialization of Picard groups. We refer to \cite{Kol20}*{Thm.~74} for this refinement of Lemma \ref{lem:RS}. It would also be interesting to extend Lemma~\ref{lem:RS} to positive characteristic because this may be useful for establishing further cases of Conjecture \ref{main-conj}. For instance, we could then weaken the assumption on $k$ in this section: we could let it be a field that is not a subfield of any $\ov{\bF}_p$.

%\epf

 The following is the promised topological criterion for ampleness.

\begin{prop}[\cite{Kol20}*{Lem.~67}] \label{prop:ample-crit}
Suppose that $\dim X \ge 2$. An irreducible divisor $H \subset X$ is ample if and only if for every effective divisor $D \subset X$ and distinct closed points $x, x' \in X \setminus D$, there is an effective divisor $H' \subset X$ with 
\[
\abs{H \cap D} = \abs{H' \cap D}  \qxq{and} x \in H' \qxq{but} x' \not\in H'.
\] 
\end{prop}

\begin{proof}[Sketch of proof]
To begin with the simpler direction, we assume that $H$ is ample, replace it by a multiple to assume that $H$ is Cartier with associated very ample line bundle $\sL$, and fix a $D$ and $x, x' \in X \setminus D$. By \cite{EGAIII1}*{2.2.4}, for some $n > 0$ there is an $s \in \Gamma(X, \sL^{\tensor n})$ that vanishes at $x$, does not vanish at $x'$, and is such that the vanishing locus of $s|_D$ is $H \cap D$. We can take $H'$ to be the vanishing locus of $s$. 

For the converse, we make a simplifying assumption that $\dim X \ge 3$ (for $\dim X = 2$ one needs a refinement of Lemma \ref{lem:RS}). To argue that $H$ is ample, we will use Kleiman's criterion \cite{Kle66}*{Ch.~III, Thm.~1 (i)$\Leftrightarrow$(iv) on p.~317}, according to which it suffices to show that for all distinct closed points $x, x' \in X$, there exist an integer $n > 0$ and an effective divisor $\wt{H}$ such that $\wt H \sim nH$ and $x \in \wt H$ but $x' \not\in \wt H$ (this will simultaneously prove that some $nH$ is basepoint free, so is also Cartier, as we require of ample divisors). Since $X$ is projective, Lemma \ref{lem:RS} and the Bertini theorem applied to the irreducible components of $H_{\ov{k}}$ supply a normal effective divisor $D \subset X$ not containing $x$, $x'$, or any generic point of $H$ such that $H \cap D$ is irreducible and $\Cl(X) \hra \Cl(D)$. By applying the assumption to this $D$, we find an effective divisor $H' \subset X$ with $\abs{H \cap D} = \abs{H' \cap D}$  and $x \in H'$ but $x' \not\in H'$. Since $H \cap D$ is irreducible, this equality of topological spaces means that $n H|_D \sim n' H'|_D$ for some $n, n' > 0$. The injectivity of $\Cl(X) \hra \Cl(D)$  then implies that $nH \sim n'H'$, and it remains to set $\wt{H} \ce n' H'$.
\end{proof}

Proposition \ref{prop:ample-crit} allows us to topologically recognize irreducible ample divisors on $X$. Granted this, the following proposition then expresses the linear similarity relation $\sim_{\mathrm{s}}$ between such divisors purely in terms of the topological space $\abs{X}$.

%Granted this topological recognition of irreducible ample divisors, the following result characterizes linear similarity between such divisors in purely topological terms.

\begin{prop}[\cite{Kol20}*{Lem.~68}] \label{prop:lin-sim}
Suppose that $\dim X \ge 3$. Irreducible divisors $H_1, H_2 \subset X$ with $H_1$ ample are linearly similar if and only if for any disjoint, irreducible, closed subsets $Z_1, Z_2 \subset X$ of dimension $\ge 1$ there is an irreducible divisor $H' \subset X$ with 
\[
\abs{H_1 \cap Z_1} = \abs{H' \cap Z_1} \qxq{and} \abs{H_2 \cap Z_2} =  \abs{H' \cap Z_2}.
\]
\end{prop}

\begin{proof}[Sketch of proof]
To begin with the simpler direction, we assume that $n_1H_1 \sim n_2 H_2$ for some nonzero $n_1, n_2$ and fix $Z_1, Z_2$ as in the statement. The $n_i$ must have the same sign: otherwise $\sO(mH_1)$ and $\sO(-mH_1)$ would have nonzero global sections for every large, sufficiently divisible $m > 0$. Thus, we may assume that $n_1, n_2 >~\!\!0$. After replacing $n_1$ and $n_2$ by $nn_1$ and $nn_2$ for a large $n > 0$, we then combine \cite{EGAIII1}*{2.2.4} and the Bertini theorem \cite{Jou83}*{6.10} to find a global section of $\sO(n_1H_1) \simeq \sO(n_2H_2)$ whose vanishing locus is an  irreducible ample divisor $H'$ with the desired properties (and even such that the intersection of $H'_{\ov{k}}$ with every irreducible component of $X_{\ov{k}}$ is irreducible). 

For the converse, we make a simplifying assumption that $\dim X \ge 5$ (to improve to $\dim X \ge 3$ one again needs a refinement of Lemma \ref{lem:RS})---this time the assumption is more serious because the $\dim X \ge 5$ case does not suffice for Theorem~\ref{thm:main}. Letting $H_1, H_2$ be irreducible ample divisors as in the statement, we iteratively apply Lemma~\ref{lem:RS} (with the Bertini theorem) to build disjoint, irreducible, normal closed subschemes $Z_1, Z_2 \subset X$ that are complete intersections of dimension $2$ such that $H_1 \cap Z_1 \subset Z_1$ and $H_2 \cap Z_2 \subset Z_2$ are irreducible divisors and the following restriction maps are injective: 
\[
\Cl(X) \hra \Cl(Z_1) \qxq{and} \Cl(X) \hra \Cl(Z_2).
\]
Since the intersections $H_1 \cap Z_1$ and $H_2 \cap Z_2$ are irreducible, these injections and the displayed equalities involving $H'$ ensure that $n_1 H_1 \sim n_1' H'$ and $n_2H_2 \sim n_2'H'$ for some $n_i, n_i' > 0$. It then follows that $n_1n_2' H_1 \sim n_1'n_2 H_2$, so that $H_1$ and $H_2$ are linearly similar, as desired. 
\end{proof}

Propositions \ref{prop:ample-crit} and \ref{prop:lin-sim} jointly carry out the first reconstruction step promised in~\S\ref{sec:master}:
\[
\abs{X} \leadsto (\abs{X}, \sim_{\mathrm{s}} \x{of irreducible ample divisors}). 
\]
They also topologically determine complete intersection subvarieties as follows. 

\begin{coro} \label{cor:H-sci}
Suppose that $\dim X \ge 3$ and let $H \subset X$ be an irreducible ample divisor. The topological space $\abs{X}$ alone determines the collection of those closed subsets $Z \subset \abs{X}$ that are \emph{set-theoretic complete $H$-intersections}, i.e.,~for which there are irreducible divisors $H_i \sim_{\mathrm{s}} H$ for $i = 1, \dotsc, r$ with $r = \codim(Z, X)$ such that 
\[
Z = \abs{H_1 \cap \dotsc \cap H_r}.
\]
\end{coro}

\begin{proof}
%Since $X$ is projective, similarly to after Lemma \ref{lem:RS}, we must have $n_i > 0$. Thus, the $H_i$ are automatically ample. 
Propositions \ref{prop:ample-crit} and \ref{prop:lin-sim} imply that $\abs{X}$ alone determines the property of $H$ being ample, as well as the linear similarity relation $H_i \sim_{\mathrm{s}} H$. 
\end{proof}

We will call such a closed subscheme $H_1 \cap \ldots \cap H_r \subset X$ a \emph{complete $H$-intersection}. The requirement that the $H_i$ be irreducible and only linearly similar (as opposed to linearly equivalent) to $H$ makes this definition slightly nonstandard, but it is convenient because Propositions \ref{prop:ample-crit} and \ref{prop:lin-sim} only concern irreducible divisors. Any positive-dimensional complete $H$-intersection $H_1 \cap \dotsc \cap H_r$ is automatically geometrically connected by the Lefschetz hyperplane theorem \cite{SGA2new}*{XII, 3.5}, and the same then also holds for set-theoretic complete $H$-intersections.

%%%%%%%%%%%%%%%%%%%%%%%%%%%%%%%%%

\section{Recovering $\bQ$-linear equivalence of ample divisors} \label{sec:linear-ample}

\textsc{Notation.} \emph{In this section, we let $X$ be a normal, geometrically integral, projective variety over field $k$ of characteristic $0$ and let $H \subset X$ be an irreducible ample divisor. }

\msk

To prepare for topological recognition of linear equivalence of divisors, for now we continue to restrict to irreducible ample divisors and show how to recognize $\bQ$-linear equivalence between them. This refines the result presented in the previous section because $\bQ$-linear equivalence $\sim_{\bQ}$ is a finer relation than linear similarity $\sim_{\mathrm{s}}$. In addition, it involves techniques that will also be relevant later, such as topological recognition of reduced $0$-dimensional intersections that we discuss in Proposition \ref{prop:restr-linking} below. A key notion behind these techniques is that of (topological) linkage defined as follows.

%Topological  is simpler in the case when these divisors are ample and irreducible---after all, we already saw in \S\ref{sec:ample} how to recognize linear similarity of such divisors. 

\begin{defi} \label{def:free-linking}
Let $Y_1, Y_2 \subset X$ be integral closed subschemes with $\dim(Y_1 \cap Y_2) = 0$.  Irreducible divisors $H_1, H_2 \subset X$ with $H_1 \sim_{\mathrm{s}} H_2 \sim_{\mathrm{s}} H$ are \emph{$H$-linked for $Y_1$ and $Y_2$} if some irreducible divisor $\wt{H} \subset X$ with $\wt{H} \sim_{\mathrm{s}} H$   satisfies
\[
|\wt{H} \cap Y_1| = |H_1 \cap Y_1| \qxq{and} |\wt{H} \cap Y_2| = |H_2 \cap Y_2|.
\]
The \emph{$H$-linking of $Y_1$ and $Y_2$ is free} if, for some finite set of closed points $\Sigma \subset Y_1 \cup Y_2$, all $H_1$ and $H_2$ as above that are disjoint from $\Sigma$ are $H$-linked for $Y_1$ and $Y_2$.
\end{defi}

By Propositions \ref{prop:ample-crit} and \ref{prop:lin-sim}, if $\dim X  \ge 3$, then these notions depend only on $\abs{X}$. They topologically encode reducedness of $0$-dimensional schematic intersections as follows.

\bprop[\cite{Kol20}*{Prop.~81}] \label{prop:free-linking}
Let $Y_1$ and $Y_2$ be as in Definition \uref{def:free-linking} and suppose that $\dim X \ge 3$, $\dim Y_1 \ge 2$, $\dim Y_2 \ge 1$, and $Y_1$ is geometrically connected (for instance, a set-theoretic complete $H$-intersection). Then the $H$-linking of $Y_1$ and $Y_2$ is free if and only if $Y_1 \cap Y_2$ is reduced with $\Gamma(Y_i, \sO) \isomto \Gamma(Y_1 \cap Y_2, \sO)$ for some $i$.
\eprop

\begin{proof}[Sketch of proof]
Since $\sO(H)$ is ample, its global sections on $Y_1 \cup Y_2$ lift to $X$ after possibly replacing them by powers. The $H_i$ and $\wt{H}$ that appear in the definition of free $H$-linking correspond to some  $s_i \in \Gamma(X, \sO(n_i H))$ and $\wt{s} \in \Gamma(X, \sO(\wt{n}H))$. Thus, in essence, the question of free $H$-linking of $Y_1$ and $Y_2$ is that of patching the sections $s_i|_{Y_i}$ along $Y_1 \cap Y_2$ to glue some of their powers to a section over $Y_1 \cup Y_2$. We may adjust the $s_i|_{Y_i}$ by global units, so the glueing is intimately related to the restriction map
\[
\Gamma(Y_1, \sO)^\times \times \Gamma(Y_2, \sO)^\times \ra \Gamma(Y_1 \cap Y_2, \sO)^\times.
\]
The analysis of this map eventually gives the claim, see \emph{loc.~cit.}~for details.
\end{proof}

\begin{rema}
In this proposition, the geometric connectedness of $Y_1$ ensures that $k \isomto \Gamma(Y_1, \sO)$. For $X$ of dimension $\ge 4$, this then leads to a topological criterion for recognizing when a $k$-smooth closed point $x \in X$ is $k$-rational because one may realize such an $x$ as the schematic intersection $Y_1 \cap Y_2$ of set-theoretic complete $H$-intersections $Y_i$ with $\dim Y_1 = 2$ and $\dim Y_2 = 1$, see \cite{Kol20}*{Cor.~82} and Corollary \ref{cor:H-sci} above. 
\end{rema}

More generally, by building on the idea of analyzing free $H$-linking of set-theoretic complete $H$-intersections, Koll\'{a}r is able to topologically recognize isomorphy of $0$-dimensional closed subschemes lying in the smooth locus as follows.

%one may recognize isomorphy of $0$-dimensional reduced closed subschemes in the smooth locus as follows.

\bprop[\cite{Kol20}*{\S84}] \label{prop:recognize-iso}
Suppose that $\dim X \ge 4$ and let $Z_1, Z_2 \subset X^\sm$ be $0$-dimensional, reduced closed subschemes. Then $(|X|, \sim_{\mathrm{s}} \!\x{of irreducible ample divisors})$ alone determines whether or not $Z_1$ and $Z_2$ are isomorphic as $k$-schemes. 
\eprop

In the setting of Proposition \ref{prop:free-linking}, it is also possible to topologically determine when $Y_1 \cap Y_2$ is reduced (without the additional condition on the global sections) as follows.

\bprop \label{prop:restr-linking}
Suppose that $\dim X \ge 3$ and let $Y \subset X$ be an irreducible, geometrically connected closed $k$-subvariety. For irreducible, geometrically connected closed $k$-subvarieties $Y', Z \subset X$ such that 
\[
|Y \cap Z| = |Y' \cap Z| \qx{and this intersection is $0$-dimensional,}
\]
consider the following condition (that is topological by Propositions~\ref{prop:ample-crit} and \ref{prop:lin-sim}):
\be\tag{$\bigstar$} \label{eqn:min-restr}
\parbox{30em}{there is a finite set of closed points $\Sigma \subset Y \cup Y' \cup Z$ such that  all irreducible divisors $H_1, H_2 \subset X$ disjoint from $\Sigma$ with $H_i \sim_{\mathrm{s}} H$ that are $H$-linked for $Y'$ and $Z$ are also $H$-linked for $Y$ and $Z$.}
\ee

\benum
\m {\upshape(}\cite{Kol20}*{Cor.~89}{\upshape)}. \label{RL-a}
If $\dim Y \ge 2$ and $Z$ satisfies $\dim X - 3 \ge \dim Z \ge 1$ and is such that \eqref{eqn:min-restr} holds for every $2$-dimensional, irreducible complete $H$-intersection $Y'$, %(as after Corollary~\ref{cor:H-sci}), 
then $Y \cap Z$ is reduced.

\m {\upshape(}\cite{Kol20}*{Cor.~93}{\upshape)}. \label{RL-b}
If $\dim Y \ge 3$ and $y \in Y$ is a closed point such that $X$ is $k$-smooth at $y$, then $Y$ is $k$-smooth at $y$ if and only if there is an irreducible complete $H$-intersection $Z \subset X$ of codimension $\dim Y$ such that $Y \cap Z$ is $0$-dimensional, contains $y$, and \eqref{eqn:min-restr} holds whenever $Y'$ is an irreducible complete $H$-intersection. 
\eenum
\eprop

In \ref{RL-a}, if $Y \cap Z$ is reduced, then $Y \cap Z \subset Y' \cap Z$, so a patching of global sections of powers of $\sO(H)$ that gives rise to an $H$-linking of $H_1$ and $H_2$ for $Y'$ and $Z$ also gives a required patching with $Y$ in place of $Y'$ (compare with the sketch of proof for Proposition \ref{prop:free-linking}). Thus, the main part is the converse, for which we refer to \emph{loc.~cit.} The role of the assumption on $\dim Z$ is to ensure, via the Bertini theorem, that there are many possible $Y'$ with $|Y' \cap Z| = |Y \cap Z|$: the intersection of all such $Y'$ is $(Y \cap Z)^\red$.

In addition to topologically recognizing $k$-points and reducedness of $0$-dimensional intersections as above, Koll\'{a}r determines equality of intersection numbers as follows.

\bprop[\cite{Kol20}*{Cor.~96}] \label{prop:intersection-numbers}
Suppose that $\dim X \ge 2$. For prime divisors $D_1, \dotsc, D_n \subset X$ and rational numbers $q_{ij} \in \bQ_{> 0}$ with $1 \le i, j \le n$, we have
\[
D_i . H^{\dim X - 1} = q_{ij} \cdot D_j . H^{\dim X - 1} \qxq{for all} 1 \le i, j \le n
\]
%we have $D_1 . H^{\dim X - 1} = \dotsc = D_n . H^{\dim X - 1}$ 
if and only if for some closed $Z_0 \subset X$ of codimension $\ge 2$ and every closed $Z \subset X$ of codimension $\ge 2$ containing $Z_0$  there is a $1$-dimensional, irreducible complete $H$-intersection $C \subset X$ disjoint from $Z$ such that each $C \cap D_i$ is a disjoint union of $m_i$ copies of $\Spec(K)$ for a finite field extension $K/k$ that does not depend on $i$, and~$q_{ij} = \f{m_i}{m_j}$.
\eprop

\begin{proof}[Sketch of proof]
The `only if' follows from the definitions: indeed, the intersection number $D_i . H^{\dim X - 1}$ is read off from the schematic intersection $C \cap D_i$. For the `if,' one first reduces to $X$ being a surface by cutting it by irreducible ample divisors that are linearly similar to $H$ and constructed via the Bertini theorem. Then the $D_i$ are curves and one seeks a $C$ cut out by some nonzero $s \in H^0(X, \sO(mH))$ for a large $m > 0$ (such an $s$ would lift to a section of $\sO(mH)$ defined over the original $X$). By considering $\sO(mH)|_{\bigcup_{i = 1}^n D_i}$ instead, it  suffices to find a global section $s_0$ of this sheaf for a large $m > 0$ such that $\{s_0 = 0\}$ is disjoint from $Z$ and $D_i \cap D_j$ for distinct $D_i$ and $D_j$ and is a union of copies of $\Spec(K)$ for some finite extension $K/k$. The key input to finding this $s_0$ is a variant of a result of Poonen \cite{Poo01}, according to which, for any quasi-finite, generically \'{e}tale morphism $\pi \colon \bigcup_{i = 1}^n D_i \ra \bP^1_k$, there are infinitely many separable closed points $p \in \bP^1_k$ whose $\pi$-preimage is a reduced disjoint union of copies of $p$.
\end{proof}

In practice, $Z_0$ is the nonsmooth locus $X \setminus X^\sm$. However, $X \setminus X^\sm$ need not \emph{a priori} be determined by the topological space $\abs{X}$, so, to get around this, one allows larger $Z$ while still retaining the smoothness of $X \setminus Z$. In particular, Propositions \ref{prop:recognize-iso} and \ref{prop:restr-linking}~\ref{RL-a} show that, for $X$ of dimension $\ge 4$, Proposition~\ref{prop:intersection-numbers} gives a purely topological criterion for determining the ratios between the intersection numbers $D_i . H^{\dim X - 1}$. 

\begin{coro} \label{cor:intersection-numbers}
Suppose that $\dim X \ge 4$. For prime divisors $\{D_\lambda\}_{\lambda \in \Lambda}$ on $X$, the topological space $\abs{X}$ alone determines the ratios between the $D_\lambda . H^{\dim X - 1}$. %are independent of $\lambda$.
\end{coro}

\begin{proof}[Sketch of proof]
It suffices to apply the reasoning above to every pair $\{D_\lambda, D_{\lambda'}\}$.
\end{proof}

In a similar vein, Proposition~\ref{prop:intersection-numbers} implies the following criterion for recognizing $\bQ$-linear equivalence of irreducible ample divisors.

\begin{coro}[\cite{Kol20}*{Claim 97.5}] \label{cor:ample-lin-eq}
Suppose that $\dim X \ge 2$. For irreducible ample divisors $H_1, H_2 \subset X$, we have 
\[
H_1 \sim_\bQ H_2 \qxq{if and only if both}  H_1 \sim_{\mathrm{s}} H_2  \qxq{and} H_1. H^{\dim X - 1} = H_2 . H^{\dim X - 1}.
\]
\end{coro}

\begin{proof}[Sketch of proof]
The `only if' is clear because intersection numbers are insensitive to linear equivalence. For the `if,' suppose that $n_1H_1 \sim n_2 H_2$ and use Proposition \ref{prop:intersection-numbers} to find a $1$-dimensional, irreducible complete $H$-intersection $C \subset X$ such that $C \cap H_1$ and $C\cap H_2$ are $0$-dimensional, reduced, and $k$-isomorphic. This $k$-isomorphy and the assumed equality of intersection numbers imply that $n_1 = n_2$, so that $H_1 \sim_\bQ H_2$. 
\end{proof}

In the case when $\dim X \ge 4$, the criterion given by Corollary \ref{cor:ample-lin-eq} is topological thanks to Corollary \ref{cor:intersection-numbers} and Propositions \ref{prop:ample-crit} and \ref{prop:lin-sim}.

%%%%%%%%%%%%%%%%%%%%%%%%%%%%%%%%%%

\section{Topological pencils of divisors} \label{sec:pencils}

\textsc{Notation.} \emph{In this section, we let $X$ be a geometrically normal, geometrically integral, positive-dimensional projective variety over a field $k$. }

\msk

The basic idea for topologically recognizing linear equivalence between general divisors on $X$ is to first make them ample by adding a multiple of some ample divisor and to then place them into linear pencils with a common member whose general members are irreducible. In some sense  this strategy achieves a reduction to the case of irreducible ample divisors considered in \S\S\ref{sec:ample}--\ref{sec:linear-ample}, and the key for carrying it out is to topologically describe families of divisors  that end up constituting the desired pencils. The central notion is that of a topological pencil that we are going to examine in this section.

%In particular, one needs to be able to topologically recognize those  that constitute a linear pencil, so we begin by discussing this aspect.

\begin{defi} \label{def:base-locus}
For a topological space $T$, a \emph{base locus} of an infinite collection $\cP = \{D_\lambda\}_{\lambda \in \Lambda}$ of subsets $D_\lambda \subset T$, denoted by $\Base(\cP)$, is a closed subset $B \subset T$ such that, for some $\Lambda' \subset \Lambda$ whose complement is finite, we have
\[
D_\lambda \cap D_{\lambda'} \subseteq B \qxq{for all distinct} \lambda, \lambda' \in \Lambda,  \qxq{with} D_\lambda \cap D_{\lambda'} = B \qxq{when} \lambda, \lambda' \in \Lambda'.
\]
\end{defi}

Due to the last requirement, the base locus $\Base(\cP)$ is unique if it exists.

\begin{defi}[\cite{Kol20}*{Def.~105}]
A \emph{topological pencil} on $X$ is a set $\cP = \{D_\lambda\}_{\lambda\in \Lambda}$ of reduced divisors $D_\lambda \subset X$ such that 
\begin{itemize}
\item
all but finitely many $D_\lambda$ are irreducible\uscolon

\item
the $D_\lambda$ jointly cover all the closed points of $X$\uscolon

\item
$B \ce \Base(\cP)$ exists, is of codimension $\ge 2$ in $X$, and each $D_\lambda \setminus B$ is connected. 
\end{itemize}
A topological pencil $\{D_\lambda\}_{\lambda \in \Lambda}$ is \emph{ample} if all but finitely many $D_\lambda$ are ample divisors.
\end{defi}

Since $B$ is of codimension $\ge 2$, the divisors $D_\lambda$ in a topological pencil are connected and pairwise have no common irreducible components. The notions of a topological pencil and of its base depend only on the topological space $\abs{X}$ and, if $\dim X \ge 2$, then, by Proposition~\ref{prop:ample-crit}, so does the ampleness of such a pencil. The following is the principal source of topological pencils.

\begin{exem} \label{eg:alg-pencil}
Let $C$ be a normal, projective, integral $k$-curve  and let $\pi \colon X \dashrightarrow C$ be a dominant rational $k$-morphism whose maximal locus of definition is $X \setminus B$ for a closed subset $B \subset X$. Since $X$ is normal and $C$ is projective and nonsingular, $B$ is of codimension $\ge 2$ in $X$ and $\pi|_{X \setminus B}$ is flat. Let $\cP = \{D_\lambda\}$ be the collection of the closures in $X$ of the reduced connected components of $(\pi|_{X\setminus B})\i(c)$ for a variable closed point $c \in C$.  We will now show that $\cP$ is a topological pencil on $X$ with base locus $B$. %We will call an element $D \in \cP$ a \emph{true member}

The rational map $\pi$ factors through the normalization $\wt{C}$ of $C$ in the maximal algebraic subextension of $k(X)/k(C)$ and this factorization has the same maximal locus of definition $X\setminus B$. By replacing $C$ by $\wt{C}$, we do not change $\cP$ and may assume that $k(C)$ is algebraically closed in $k(X)$, so that, by \cite{EGAIV2}*{4.5.9}, the generic fiber of $\pi$ is geometrically irreducible. It then follows from \cite{EGAIV3}*{9.7.7} that all but finitely many $D_\lambda$ are irreducible.

Consider the closure $\ov{X} \subset X \times_k C$ of the graph of $\pi|_{X\setminus B}$, which inherits $k$-morphisms 
\[
\xymatrix{
& \ov{X} \ar[ld]_-{b} \ar[rd]^-{\ov{\pi}} \\
X \ar@{-->}[rr]^{\pi} & & C
}
\]
such that $b$ is an isomorphism over $X \setminus B$ and both $\ov{\pi}$ and $b$ are proper. The locus $U$ of $X$ over which $b$ has finite fibers is open (see \cite{SP}*{01TI}), so, since $X$ is normal and $b$ is birational, the finite map $b|_{b\i(U)}$ is an isomorphism. In particular, $U = X\setminus B$ and $B$ consists precisely of the points $x \in X$ such that $\dim(b\i(x)) > 0$. Since $\ov{X} \subset X \times_k C$, for closed such $x$ the map $b\i(x) \ra C$ is finite surjective and $\dim(b\i(x)) = 1$.

The closure in $X$ of a reduced connected component of $(\pi|_{X\setminus B})\i(c)$ is the $b$-image  of the closure in $\ov{X}$ of the corresponding reduced connected component of $(\ov\pi|_{\ov X\setminus b\i(B)})\i(c)$. Moreover, since $C$ is nonsingular, $\ov{\pi}$ is flat and its closed fibers are purely of dimension $\dim X - 1$. Thus, we conclude from the previous paragraph that the $D_\lambda$ jointly cover all the closed points of $X$ and that the base locus of $\cP$ is precisely $B$ (in Definition~\ref{def:base-locus} choose $\Lambda'$ to consist of those closed points $c \in C$ such that the $c$-fibers of $\pi$ and $\ov{\pi}$ are geometrically irreducible). % $\cP$ is a topological pencil on $X$ with base locus $B$, as claimed.
%In particular, since $B$ is of codimension $\ge 2$ in $X$ and $b\i(B)$ lies in $B \times_k C$, the open $X\setminus B \subset \ov{X}$ is $C$-fiberwise dense. It follows that every closed point of $\ov{X}$ lies in the closure of a reduced connected component of $(\ov\pi|_{\ov X\setminus b\i(B)})\i(c)$ for some $c$, so that the $D_\lambda$ jointly cover the closed points of $X$. 
% Thus, $U = X \setminus B$, to the effect that each closed point in $B$ lies in infinitely many $D_\lambda$. Since, evidently, $D_\lambda \cap D_\lambda' \subset B$ for all $\lambda \neq \lambda'$, this means that $\Base(\{D_\lambda\}) = B$, as claimed.
\end{exem}

\begin{defi} 
A topological pencil $\cP$ on $X$ is \emph{algebraic} (resp.,~\emph{rational}; resp.,~\emph{linear}) if it is associated to some $C$ and $\pi$ as in Example \ref{eg:alg-pencil} (resp.,~with $C_{\ov{k}} \simeq \bP^1_{\ov{k}}$; resp.,~with $C \simeq \bP^1_k$); an algebraic $\cP$ is \emph{noncomposite} if $\wt{C} = C$ in Example \ref{eg:alg-pencil}, that is, if every finite morphism $C' \ra C$ of normal, projective, integral $k$-curves through which $\pi$ factors is an isomorphism, equivalently, if $k(C)$ is algebraically closed in $k(X)$.
\end{defi}

The following example relates topological pencils and linear equivalence. 

\begin{exem} \label{eg:linear-pencil}
Let $\sL$ be a line bundle on $X$ and let $s_0, s_\infty \in H^0(X, \sL)$ be nonzero global sections. The vanishing loci $D_0 \ce \{s_0 = 0\}$ and $D_\infty \ce \{s_\infty = 0\}$ are linearly equivalent divisors on $X$. If $D_0$ and $D_\infty$ have no common irreducible components, then $s_0$ and $s_\infty$ span a linear pencil whose base locus is $\{s_0 = s_\infty = 0\}$: indeed, to relate to Example \ref{eg:alg-pencil}, it suffices to note that $s_0$ and $s_\infty$ determine a rational map $\pi\colon X \dashrightarrow \bP^1_k$ whose maximal locus of definition is $X \setminus \{s_0 = s_\infty = 0\}$. 

Conversely, for any pair of $k$-points $c, c' \in \bP^1_k$, there is a rational function $f \in k(t)$ that vanishes to order one at $c$ and has a simple pole at $c'$, so any two $k$-fibers of a linear topological pencil $\cP$ are linearly equivalent on $X$. 
\end{exem}

To utilize this example, we need to recognize linearity of algebraic pencils. %This is achieved in

%As a first step, the following lemma helps one recognize rationality of algebraic pencils.

%As far as rationality of algebraic pencils is concerned, this is accomplished through the following lemma. 

\begin{lemm}[\cite{Kol20}*{Lem.~109}] \label{lem:lin-crit}
Suppose that $k$ is perfect and let $\cP$ be an algebraic topological pencil on~$X$. 
\benum
\m \label{LC-a}
If $\Base(\cP)$ meets the $k$-smooth locus $X^\sm \subset X$, then $\cP$ is rational. 

\m \label{LC-b}
If $\cP$ is rational and $X^\sm(k) \neq \emptyset$, then $\cP$ is linear. 

\m \label{LC-c}
If $\cP$ is rational and the $k$-smooth locus of some $D \in \cP$ contains a nonempty open of some geometrically irreducible closed $k$-subvariety $Y \subset \Base(\cP)$, then $\cP$ is~linear.

%, associated to some $\pi \colon X \dashrightarrow C$ as in Example \ref{eg:alg-pencil} with $C_{\ov{k}} \simeq \bP^1_{\ov{k}}$, and there is a closed point $c \in C$ such that $\pi\i(c)$ is reduced and its closure in $X$ 

\eenum
\end{lemm}

\bpf 
Let $C$, $\pi$, and $\ov{X}$ be as in Example \ref{eg:alg-pencil} and let $\wt{\ov{X}}$ be the normalization of $\ov{X}$. The role of the perfectness of $k$ is to ensure that the normal $k$-curve $C$ is $k$-smooth.
%Let $x \in \Base(\cP) \cap X^\sm$ be a closed point and 
%Let $\pi$, $b$, and $\ov{X}$ be as in Example~\ref{eg:alg-pencil}.
\benum
\m
By a result of Abhyankar \cite{Kol96}*{VI.1.2}, the positive dimensional fibers of a proper modification $Y' \ra Y$ of excellent, normal schemes with $Y$ regular contain nonconstant rational curves (see also \cite{Bha12a}*{Rem.~4.5}). We apply this to the restriction of the morphism $\wt{\ov{X}} \ra X$ to $X^\sm$: by using the assumption on $\Base(\cP)$, we conclude that each positive-dimensional fiber of $\ov{X} \ra X$ receives a nonconstant morphism from a rational curve. It then follows from Example \ref{eg:alg-pencil} that $C_{\ov{k}}$ also receives such a morphism, so that $C_{\ov{k}}\simeq \bP^1_{\ov{k}}$. Thus, $\cP$ is rational, as desired.

\m
If $X^\sm(k)\neq \emptyset$, then the Lang--Nishimura theorem \cite{Poo17}*{3.6.11} implies that $C(k) \neq \emptyset$. Since, by assumption, $C_{\ov{k}} \simeq \bP^1_{\ov{k}}$, we then 	conclude that $C \simeq \bP^1_k$. 

\m
Suppose that $D$ arises from a closed point $c \in C$ as in Example \ref{eg:alg-pencil}.  It suffices to argue that $[k(c) : k] = 1$, since then $C(k) \neq \emptyset$ and $C \simeq \bP^1_k$ (compare with \ref{LC-b}). Since $k$ is perfect, Example \ref{eg:alg-pencil} applied over $\ov{k}$ shows that %the closure of every fiber of $\pi$ over 
$D_{\ov{k}}$ is a $\Gal(\ov{k}/k)$-orbit of $[k(c) : k]$ closures of connected components of fibers of $\pi$ over $\ov{k}$. Thus, if $D$ is $k$-smooth at the generic point of some geometrically irreducible $k$-subvariety $Y \subset \Base(\cP)$, then the closures of distinct connected components of fibers of $\pi$ over $\ov{k}$ cannot simultaneously contain a nonempty open of $Y_{\ov{k}}$, so $[k(c) : k] = 1$.
\qedhere
\eenum
\epf

We turn to the key question of topologically recognizing when a topological pencil is algebraic. The most basic example is the following case of an empty base locus.

\begin{exem}
Suppose that $k$ is algebraically closed. Then, by \cite{BPS16}*{Thm.~1.1}, every topological pencil $\{D_\lambda\}_{\lambda \in \Lambda}$ whose base locus is empty, in other words, such that $D_\lambda \cap D_{\lambda'} = \emptyset$ for $\lambda \neq \lambda'$, is algebraic. Indeed, \emph{loc.~cit.}~says that there are a smooth, projective $k$-curve $C$ and a surjective $k$-morphism $\pi\colon X \ra C$ with connected fibers such that each $D_\lambda$ is contained in a (closed) fiber of $\pi$. Since the $D_\lambda$ jointly cover the closed points of $X$, the set of closed fibers of $\pi$ is then precisely $\{D_\lambda\}_{\lambda \in \Lambda}$.
\end{exem}

%To be able to employ it for this purpose, however, we first need to be able to topologically recognize algebraicity and linearity of topological pencils, as well as to topologically recognize true members of linear pencils. 

To proceed beyond empty base loci, it is useful to first note that algebraic topological pencils are determined by infinitely many members as follows. In essence, this is the basic reduction mechanism for reaching irreducible ample divisors from general divisors.

\begin{lemm} \label{lem:infinite-intersection}
For topological pencils $\cP$ and $\cP'$ on $X$ with $\cP'$ algebraic, if the set $\cP \cap \cP'$ of those divisors $D \subset X$ that belong to both $\cP$ and $\cP'$ is infinite, then $\cP = \cP'$. 
\end{lemm}

\begin{proof}
The infinitude of $\cP \cap \cP'$ implies that $\Base(\cP) = \Base(\cP')$, so we let $B$ be this common base locus and let $\pi \colon X \dashrightarrow C$ be a dominant rational morphism that gives rise to $\cP'$ as in Example \ref{eg:alg-pencil}. As in that example, $X\setminus B$ is the maximal locus of definition of $\pi$ and we may assume that the generic fiber of $\pi$ is geometrically irreducible. The complements $D \setminus B$ for $D \in \cP$ are connected and pairwise disjoint, so the infinitude of $\cP \cap \cP'$ ensures that each $D \setminus B$ lies in a single fiber of $\pi$. Since topological pencils cover the closed points of $X$, it then follows that $D \in \cP'$ and that $\cP = \cP'$, as desired.
\end{proof}

The following is a topological criterion for algebraicity of topological pencils. 

\bprop[\cite{Kol20}*{Prop.~107}] \label{prop:pencil-algebraic}
Suppose that $k$ is infinite. %and let $H \subset X$ be an irreducible ample divisor. 
A topological pencil $\cP = \{D_\lambda\}_{\lambda \in \Lambda}$ on $X$ is algebraic if and only if for some infinite subset $\Lambda' \subset \Lambda$ and every (or merely some) irreducible ample divisor $H \subset X$, the intersection numbers
\[
D_\lambda . H^{\dim X - 1} \qxq{are all equal for} \lambda \in \Lambda';
\]
thus, if $\Char(k) = 0$ and $\dim(X) \ge 4$, then the algebraicity of $\cP$ depends only~on~$\abs{X}$. 
\eprop

\begin{proof}[Sketch of proof]
The last assertion that concerns topological recognition of algebraicity of $\cP$ follows from the rest and from Proposition \ref{prop:ample-crit} and Corollary~\ref{cor:intersection-numbers}. 

For the rest, when $\dim(X) = 1$, every $\cP$ is algebraic and the claim is that $X$ has infinitely many closed points of the same degree over $k$. This holds because there is a finite $k$-morphism $X \ra \bP^1_k$ and $\bP^1(k)$ is infinite. Thus, we may assume that $\dim X \ge 2$.

We begin with  the simpler `only if' direction and assume that $\cP$ is algebraic, associated to a $\pi \colon X \dashrightarrow C$ as in Example \ref{eg:alg-pencil} such that all but finitely many fibers  of $\pi$ are geometrically irreducible.  By the $1$-dimensional case, there is an infinite set $\Lambda'$ of closed points of $C$ of the same degree over $k$ with irreducible $\pi$-fibers. By using the Bertini theorem, for every irreducible ample divisor $H \subset X$ and every $\lambda, \lambda' \in \Lambda'$ we may find a complete $H$-intersection curve $Y \subset X$ for which $Y \cap \Base(\cP) = \emptyset$, both $Y \cap D_\lambda$ and $Y \cap D_{\lambda'}$ are $0$-dimensional, and $Y$ is flat over a neighborhood of $\lambda$ and $\lambda'$ in $C$. Then 
\[
D_\lambda . H^{\dim X - 1} = D_{\lambda'} . H^{\dim X - 1}
\]
because both these intersection numbers are equal to the product of the degree of $Y$ over $C$ with the common degree of the points in $\Lambda'$. 

For the converse, 
%we first claim that the $D_{\lambda}$ for $\lambda \in \Lambda'$ are pairwise $\tau$-equivalent in the sense that the classes of their pairwise differences $D \ce D_\lambda - D_{\lambda'}$ are torsion in the N\'{e}ron--Severi group of $X$. For this, since $X$ is geometrically irreducible of dimension $\ge 2$, by \cite{SGA6}*{XIII, Thm.~4.6}, it suffices to argue that for some irreducible ample divisor $H \subset X$ we have $D.H^{\dim(X) - 1} = 0$ and $D^2.H^{\dim(X) - 2} = 0$. The first of these vanishings follows from the assumption. To get the second one (in fact, both), we first note that $mH + D$ is very ample for every large enough $m$ (see \cite{Har77}*{II, Exercise~7.5}), so that, by the assumption and the Bertini theorem, $(mH + D)^{\dim(X) - 1}. D = 0$. By then regarding the left side as a polynomial $P(m)$ in $m$ and iteratively considering the difference polynomials $P(m + 1) - P(m)$, we then induct on $i$ to get the desired
%\[
%H^{\dim(X) - 1 - i}.D^{i + 1} = 0 \qxq{for} i = 0, \dotsc, \dim(X) - 1.
%\] 
%Since the $D_\lambda$ for $\lambda \in \Lambda'$ are pairwise $\tau$-equivalent and $\NS(X)_\tors$ is finite (see \cite{SGA6}*{XIII, Thm.~5.1}), by passing to an infinite subset of $\Lambda'$ we may assume that the $D_\lambda$ for $\lambda \in \Lambda'$ are pairwise algebraically equivalent. 
%At this point, 
we fix a single $H$ and let $d$ be the common value $D_\lambda . H^{\dim X - 1}$ for $\lambda \in \Lambda'$ and consider the Chow $k$-scheme $\mathrm{Chow}_{X,\, d}$ that parametrizes those effective divisors $D \subset X$ that satisfy $D . H^{\dim X - 1} = d$, so that $\mathrm{Chow}_{X,\, d}$ is projective over $k$ (see \cite{Kol96}*{Ch.~I}). The $D_\lambda$ for $\lambda \in \Lambda'$  give infinitely many closed points on $\mathrm{Chow}_{X,\, d}$, so, since $\mathrm{Chow}_{X,\, d}$ is of finite type over $k$, their closure contains a positive-dimensional irreducible closed subscheme $C \subset \mathrm{Chow}_{X,\, d}$. Consider the universal family of divisors $\pi\colon E \ra C$ base changed from $\mathrm{Chow}_{X,\, d}$, as well as the resulting commutative diagram
\[
\xymatrix{
& E \ar[d]_-{\pi} \ar[r]^-i & X \\
\mathrm{Chow}_{X,\, d} \ar@{}[r]|-{\text{\mbox{\large$\supset$}}} &  C
}
\]
for which the $D_\lambda$ for $\lambda \in \Lambda'$ appear as fibers of $\pi$. It suffices to argue that there is a nonempty open $X^0 \subset X$ such that the $i\i(x)$ for a dense set of closed points $x \in X^0$ are singletons. Then, up to a power of Frobenius if $\Char k > 0$, the map $i$ will be birational, $C$ will be a curve by counting dimensions, $\pi$ will give rise to an algebraic topological pencil $\cP'$ on $X$ as in Example \ref{eg:alg-pencil}, and Lemma~\ref{lem:infinite-intersection} will imply the desired $\cP = \cP'$.

For the claim about $X^0$, first of all, the image of $i$ contains infinitely many distinct divisors $D_\lambda$ for $\lambda \in \Lambda'$, so the Chevalley constructibility theorem \cite{EGAIV1}*{1.8.4} implies that $i$ is dominant. We then let 
\[
X^0 \subset X \setminus \Base(\cP)
\]
be a nonempty open over which $i$ is flat (see \cite{EGAIV3}*{11.2.6 (ii)}) and let $x$ range over the closed points of $D_\lambda\cap X^0$ for $\lambda \in \Lambda'$. Suppose that for such an $x \in D_\lambda \cap X^0$ the fiber $i\i(x)$ is not a singleton. Then the divisor $i\i(D_\lambda  \cap X^0) \subset i\i(X^0)$ meets some $\pi$-fibers $E_c$ over closed points $c \in C$ such that $E_c$ is different from the $\pi$-fiber $D_\lambda$. By construction, the effective divisors $E_c$ on $X$ are all algebraically equivalent to $D_\lambda$, so, since the latter is irreducible, the nonempty intersections $E_c \cap D_\lambda$ have pure codimension $2$ in $X$. This means that the intersections $i\i(D_\lambda \cap X^0) \cap E_c$ are nowhere~dense~in~$E_c$. 

Since $i\i(D_\lambda \cap X^0) \subset i\i(X^0)$ is of pure codimension $1$, another application of the Chevalley constructibility theorem then shows that the map $i\i(D_\lambda \cap X^0) \ra C$ given by $\pi$ is dominant and its image contains a nonempty open $C^0 \subset C$. However, this is impossible: by construction of $C$, there is a closed point $c \in C^0$ such that $E_c = D_{\lambda'}$ for some $\lambda' \in \Lambda' \setminus \{ \lambda\}$, and $i\i(D_\lambda \cap X^0)$ cannot meet $E_c$ because $D_\lambda \cap D_{\lambda'}$ lies in $\Base(\cP)$, which does not meet $X^0$. Thus, $i\i(x)$ is indeed a singleton, as desired.
\end{proof}

\begin{coro}
If $k$ is uncountable (equivalently, if $\abs{X}$ uncountable), then every topological pencil on $X$ is algebraic.
\end{coro}

\begin{proof}
To see the parenthetical equivalence it suffices to note that, by Noether normalization, $\abs{X}$ is uncountable if and only if $|\bA^{\dim(X)}_k|$ is uncountable. Suppose that $\abs{X}$ is uncountable, let $\cP = \{D_\lambda\}_{\lambda \in \Lambda}$ be a topological pencil on $X$, and fix a $\lambda \in \Lambda$ and a closed point $x \in D_\lambda \setminus \Base(\cP)$. Since $\Base(\cP)$ is of codimension $\ge 2$ in $X$, by cutting $D_\lambda$ by sufficiently general hyperplanes passing through $x$ supplied by the Bertini theorem, we may build an irreducible curve $C \subset X$ that properly meets $D_\lambda$ but does not meet $\Base(\cP)$. Then $C$ meets each $D_{\lambda'}$ in finitely many points, to the effect that, since $\abs{C}$ is uncountable and the $D_{\lambda'}$ cover the closed points of $X$, the set $\Lambda$ is also uncountable. On the other hand, the N\'{e}ron--Severi group $\NS(X)$ is countable (see \cite{SGA6}*{XIII, Thm.~5.1}). Thus, there is an infinite subset $\Lambda' \subset \Lambda$ such that the $D_{\lambda'}$ for $\lambda' \in \Lambda$ are pairwise algebraically  equivalent, and so also pairwise numerically equivalent (see \cite{SGA6}*{XIII, Thm.~4.6}). Proposition~\ref{prop:pencil-algebraic} then shows that the pencil $\cP$ is algebraic. 
\end{proof}

%%%%%%%%%%%%%%%%%%%%%%%%%%%%%%%%%%

\section{Recovering linear equivalence of divisors} \label{sec:linear}

\textsc{Notation.} \emph{In this section, we let $X$ be a geometrically normal, geometrically integral, positive-dimensional projective variety over an infinite field $k$. }

\msk

We are ready to describe Koll\'{a}r's topological recognition of linear equivalence of divisors on $X$. The following lemma allows us to only consider reduced divisors. We recall from \S\ref{sec:master} that, by definition, reduced divisors are assumed to be effective. 

\begin{lemm} \label{lem:reduced-suffice}
For the subgroup $\sR \subset \Div(X)$ generated by the differences of linearly equivalent reduced divisors, every class in $\Div(X)/\sR$ is represented by a difference of reduced divisors. In particular, $\sR$ is the subgroup of all divisors linearly equivalent to~$0$.
%linear equivalence between reduced divisors on $X$ determines linear equivalence between all divisors on $X$. 
\end{lemm}

\bpf
The last assertion follows from the rest because the difference of reduced divisors lies in $\sR$ if and only if this difference is linearly equivalent to $0$. For the rest, every divisor is a sum of irreducible divisors (with multiplicities), so it suffices to show that every irreducible divisor $D \subset X$ is linearly equivalent to $D_1 - D_2$ for some reduced divisors $D_i \subset X$ that share no irreducible components with divisors $D' \subset X$ in some fixed finite set containing $D$: then $D$ and $D_1 - D_2$ will agree in $\Div(X)/\sR$ and, by iteratively applying this with $D'$ ranging over the $D_i$ from preceding steps, we will represent every class in $\Div(X)/\sR$ by a difference of reduced divisors. 

To find a desired linear equivalence $D \sim D_1 - D_2$, we first choose a very ample divisor $H \subset X$ and an $m > 0$ such that $D + mH$ is also very ample (see \cite{Har77}*{II, Exercise~7.5}). Since $k$ is infinite, the Bertini theorem \cite{Jou83}*{6.10} then supplies our desired geometrically reduced divisors $D_i$ with $D_1 \sim D + mH$ and $D_2 \sim mH$.
\epf

As for linear equivalence of reduced divisors, the following is the key criterion. 

\bprop[\cite{Kol20}*{Thm.~115}] \label{prop:main-lin-equiv}
Suppose that $k$ is perfect and $\dim X \ge 3$, and let $H \subset X$ be an irreducible ample divisor. Reduced divisors $D_1, D_2 \subset X$ are linearly equivalent if and only if for every large enough closed subset $\Sigma \subset X$ of codimension $\ge 2$ and every integral curve $C \subset X$ not in $D_1 \cup D_2 \cup \Sigma$, there are 
%at the cost of replacing $D_1$ and $D_2$ by, respectively, $D_1 + D' + H'$ and $D_2 + D' + H'$ for some  there are 
\begin{itemize}
\item
algebraic topological pencils $\cP_1$ and $\cP_2$ with $C \subset \Base(\cP_i)$;

\m
irreducible divisors $D', E' \subset X$ not in $D_1 \cup D_2$ with $E'$ ample and containing $C$;

\item
an irreducible ample divisor $E \subset X$ with $(D_i + D' + E')\setminus (E \cup \Sigma )$ %and $D_2\setminus (H \cup \Sigma )$ 
connected;

\end{itemize}
such that 
\begin{itemize}
\m
$D_i + D' + E'$ and $E$ lie in $\cP_i$ (so $E$ also contains $C$);

\m
all but finitely many closed points of $C$ lie in the $k$-smooth loci $E^\sm$ and $E'^\sm$; %of $H$ and $H'$. 

\m
$D_i + D' + E'$ and $E$ lie in the subset of those $D \in \cP_i$ for which the function $D \mapsto D . H^{\dim X - 1}$ takes the minimal value attained infinitely many times on $\cP_i$.
\end{itemize}
%the $k$-smooth loci of $H$ and $H'$ meet~$C$. 
\eprop

\begin{proof}[Sketch of proof]
We begin with the simpler `only if' direction, suppose that $D_1$ and $D_2$ are linearly equivalent, and choose $\Sigma$ to contain $X \setminus X^\sm$. We then choose an irreducible ample divisor $D'$ that shares no components with $D_1$ and $D_2$, does not contain $C$, and is such that the $\sO(D_i) \tensor_{\sO_X} \sO(D')$ are generated by global sections. This makes the reduced divisors $D_i + D'$ basepoint free, so also Cartier. Moreover, since $k$ is perfect, the integral curve $C$ is generically $k$-smooth. We may then use the Bertini theorem \cite{KA79}*{Thm.~7} (which uses the assumption on $\dim(X)$) to find an irreducible ample divisor $E'$ that contains $C$, is $k$-smooth at the generic point of $C$, properly meets every irreducible component of $D_i + D'$, does not contain any irreducible component of $\Sigma$, and is such that the $D_i + D' + E'$ are very ample (see \cite{Har77}*{II, Exercise 7.5 (d)}). Granted that we make sure (as we may) that $E'$ is sufficiently ample, we may then apply the Bertini theorem \cite{KA79}*{Thm.~7} again, this time with the ample line bundle 
\[
\sO(D_1 + D' + E') \simeq \sO(D_2 + D' + E'),
\]
to find an irreducible ample divisor $E \subset X$ with $E \sim D_i + D' + E'$ such that $E$ contains $C$, is $k$-smooth at the generic point of $C$, and does not contain any irreducible component of $D_i + D' + E'$ nor any irreducible component of the intersections between $E'$ and the irreducible components of the $D_i + D'$. The $(D_i + D' + E')\setminus(E \cup \Sigma)$ are then connected, and we let $\cP_i$ be the linear pencil spanned by $D_i + D' + E'$ and $E$ as in Example~\ref{eg:linear-pencil}. By construction, these $D'$, $E'$, $\cP_i$, and $E$ satisfy all the requirements (to check the requirement about the intersection numbers, one uses the fact that both $D_i + D' + E'$ and $E$ are $k$-fibers of the pencil and argues with a complete $H$-intersection curve $Y$ analogous to the one used in the proof of Proposition \ref{prop:pencil-algebraic}).

For the `if' direction, we may assume that $\Sigma$ contains $X \setminus X^\sm$ and use the Bertini theorem \cite{Jou83}*{Thm.~6.10} to choose $C$ to be a generically smooth, geometrically integral complete $H$-intersection. With these choices, by Lemma \ref{lem:lin-crit}~\ref{LC-a}, the pencils $\cP_i$ are rational and, by its part \ref{LC-c}, then they are even linear. At this point one uses the the condition that involves the degree function $D \mapsto D . H^{\dim X - 1}$ to check that both $D_i + D' + E'$ and $E$ are $k$-fibers of the pencil $\cP_i$ (we omit the details of this step here). It then follows from Example~\ref{eg:linear-pencil} that $D_i + D' + E' \sim E$, so that also $D_1 \sim D_2$.
\end{proof}

\begin{coro}[\cite{Kol20}*{Thm.~116}]
Suppose that $\Char(k) = 0$ and $\dim X \ge 4$. The topological space $\abs{X}$ determines linear equivalence of divisors on $X$.
\end{coro}

\begin{proof}[Sketch of proof]
By Lemma \ref{lem:reduced-suffice}, it suffices to show that $\abs{X}$ determines linear equivalence between reduced divisors. For this, we explain why the notions and conditions that appear in the linear equivalence criterion of Proposition \ref{prop:main-lin-equiv} are determined by~$\abs{X}$:
\bitem
\m
Ampleness of irreducible divisors by Proposition \ref{prop:ample-crit};

\m
Algebraicity of topological pencils by Proposition \ref{prop:pencil-algebraic};

\m
The function $D \mapsto D . H^{\dim X - 1}$ up to constant multiple by Corollary \ref{cor:intersection-numbers};

\m
All but finitely many closed points of $C$ lying in $E^\sm$ and $E'^\sm$ by Proposition~\ref{prop:restr-linking}~\ref{RL-b}: to apply it, we choose $\Sigma$ to contain $X \setminus X^\sm$ (then all but finitely many closed points of $C$ lie in $X^\sm$) and we note that $E$ and $E'$ are geometrically connected by the Lefschetz hyperplane theorem \cite{SGA2new}*{XII, 3.5}. \qedhere
\eitem
\epf

In conclusion, we have now described the second reconstruction step promised in~\S\ref{sec:master}:
\[
(\abs{X}, \sim_{\mathrm{s}} \x{of irreducible ample divisors}) \leadsto (\abs{X}, \sim \x{of effective divisors}). 
\]

%%%%%%%%%%%%%%%%%%%%%%%%%%%%%%%%%%%%%%%%

\section{Recovering the projective variety itself} \label{sec:recover-X}

\textsc{Notation.} \emph{In this section, we let $X$ be a normal, geometrically integral projective variety of dimension $\ge 2$ over an infinite field $k$. }

\medskip

The remaining  step
\[
(\abs{X}, \sim \x{of effective divisors}) \leadsto X
\]
is a result of Lieblich and Olsson presented in their preprint \cite{LO19}. In this step, the ultimate source of reconstruction is  the fundamental theorem of projective geometry %theorem of Veblen--Young type \cite{VY08} 
that characterizes projective spaces combinatorially in terms of properties of incidence between their points and lines---in effect, for projective spaces this theorem reconstructs the full structure of an algebraic variety from axiomatic combinatorial data. It is fascinating that the collinearity relation for points in projective space encodes rich scheme-theoretic structure. This is not quite unexpected, however---after all, one knows that this relation is capable of encoding, for instance, polynomial equations with integer coefficients (see \cite{Laf03}*{p.~iv, Remarques}).

The precise version of such a theorem that Lieblich and Olsson use is as follows, established via an argument that is close to E.~Artin's classical proof.

\begin{theo}[\cite{LO19}*{Thm.~4.4}] \label{thm:VY-LO}
Let $V$ and $V'$ be finite-dimensional vector spaces over infinite fields $k$ and $k'$, respectively, and let 
\[
U \subset \Gr_1(\bP(V))(k) \qxq{and} U' \subset \Gr_1(\bP(V'))(k')
\]
be collections of lines in $\bP(V)$ and $\bP(V')$ given by the sets of $k$-points and $k'$-points of some nonempty Zariski open subsets of the indicated Grassmannians. For any bijection 
\[
b\colon \bP(V)(k) \isomto \bP(V')(k') \qxq{that induces an inclusion} U \hra U',
\]
there are a field isomorphism $\iota \colon k \isomto k'$ and a $\iota$-semilinear isomorphism $V \isomto V'$ such that the induced isomorphism $\bP(V) \isomto \bP(V')$ agrees with $b$ on some Zariski open containing all the lines in $U$. 
\end{theo}

With this theorem in hand, the strategy is to apply it with 
\[
V \ce \Gamma(X, \sO(mH)) \qxq{and} V' \ce \Gamma(X', \sO(mH')) \qxq{for all} m > 0
\]
for a suitable very ample divisor $H \subset X$ and a homeomorphic to $X$ projective variety $X'$ that one seeks to show to be isomorphic to $X$ (with $H'$ being the image of $H$). Indeed, the isomorphisms $V \isomto V'$ then give isomorphisms between the graded components of the coordinate rings that appear in the projective embeddings
\[
X \cong \Proj\p{\bigoplus_{m \ge 0} \Gamma(X, \sO(mH))} \hra \Proj\p{\bigoplus_{m \ge 0} \Gamma(X, \sO(H))^{\tensor m}}
\]
and
\[
X' \cong \Proj\p{\bigoplus_{m \ge 0} \Gamma(X', \sO(mH'))} \hra \Proj\p{\bigoplus_{m \ge 0} \Gamma(X, \sO(H'))^{\tensor m}}\!.
\]
With the help of some additional arguments to make sure that the isomorphisms of the graded pieces are compatible as $m$ varies, one obtains the desired $X \isomto X'$. 

In view of this strategy, the key becomes  defining the subsets of lines $U$ and $U'$ with which to apply Theorem~\ref{thm:VY-LO}. Since $k$ is infinite, a line $\ell \subset \bP(V)$ is spanned by any two of its $k$-points, which correspond to $k^\times$-orbits of sections of $\Gamma(X, \sO(mH))$, that is, to effective divisors on $X$ linearly equivalent to $mH$. The \emph{base locus} of $\ell$ is the locus of common vanishing on $X$ of these effective divisors. Lieblich and Olsson choose $U$ to consist of all the lines that satisfy the following definition (and analogously for $U'$).

\begin{defi}
With $V \ce \Gamma(X, \sO(mH))$ as above, a line $\ell \subset \bP(V)$ is \emph{strongly definable} if there is a closed subset $Z \subset \abs{X}$ of codimension $\ge 2$ such that $\ell$ consists precisely of those effective divisors on $X$ linearly equivalent to $mH$ that contain $Z$.
\end{defi}

The following description of the set of all strongly definable lines is critical for applying Theorem~\ref{thm:VY-LO} to reconstruct the projective variety $X$. %from $(\abs{X}, \sim \x{of effective divisors})$.

\begin{lemm}[\cite{LO19}*{Lem.~5.13}]
If the linear system determined by the ample divisor $H \subset X$ is basepoint free, then the strongly definable lines in $\bP(V)$ comprise the set of $k$-points of a nonempty Zariski open of the Grassmannian $\Gr_1(\bP(V))$. 
\end{lemm}

The precise formulation of the reconstruction result that Lieblich and Olsson obtain by carrying out the strategy that we sketched in this section is as follows.

\begin{theo}[\cite{LO19}*{Prop.~6.2.5}]
For normal, geometrically integral, proper varieties $X$ and $X'$ over infinite fields $k$ and $k'$, respectively, if $X$ is projective and $\dim X \ge 2$, then any homeomorphism $\abs{X} \isomto \abs{X'}$ that matches the linear equivalence relations on effective divisors on $X$ and $X'$ underlies a scheme isomorphism $X\isomto X'$. 
\end{theo}

This result achieves the final reconstruction step
\[
(\abs{X}, \sim \x{of effective divisors}) \leadsto X,
\]
and hence completes our sketch of the proof of Theorem~\ref{thm:main}.

\bigskip

\emph{Acknowledgements. } I thank the organizers of the S\'{e}minaire Bourbaki for their kind invitation. I thank J\'{a}nos Koll\'{a}r for patiently answering numerous questions. I thank Vi\stackon[-11pt]{\^{e}}{\d{}}n To\'{a}n H\d{o}c for an invitation to give lectures related to the material presented here.

\end{document}

%%% Local Variables:
%%% mode: latex
%%% TeX-master: t
%%% End: